\documentclass[11pt]{article}
\usepackage{amsmath,amssymb,amsthm,comment}
\usepackage{amsfonts}
\usepackage{color}
\usepackage{epsfig}
\usepackage{bm}
\newtheorem{theorem}{Theorem}[section]
\newtheorem{lemma}[theorem]{Lemma}

\newtheorem{remark}[theorem]{Remark}

\makeatletter
    
    \@addtoreset{equation}{section}
  \makeatother
\newcommand{\io}{\int_\Omega}
\newcommand{\R}{\mathbb{R}}

\newcommand{\tm}{T_{\mathrm{max}}}
\newcommand{\abs}{\\[5pt]}
\begin{document}
\title{Global existence and infinite time blow-up of classical solutions to chemotaxis systems of local sensing in higher dimensions}
\author{
Kentaro Fujie\footnote{fujie@tohoku.ac.jp}\\
{\small Tohoku University,}\\
{\small Sendai, 980-8578, Japan}
\and
Takasi Senba\footnote{senba@fukuoka-u.ac.jp}\\
{\small Fukuoka University,}\\
{\small Fukuoka, 814-0180, Japan}
\medskip
}
\date{\small\today}
\maketitle
\begin{abstract}
This paper deals with the fully parabolic chemotaxis system of local sensing in higher dimensions. 
Despite the striking similarity between this system and the Keller--Segel system, we prove the absence of finite-time blow-up phenomenon in this system even in the supercritical case. 
It means that for any regular initial data, independently of the magnitude of mass, the classical solution exists globally in time in the higher dimensional setting. 
Moreover, for the exponential decaying motility case, 
it is established that solutions may  blow up at infinite time for any magnitude of mass. 
In order to prove our theorem, we deal with some auxiliary identity as an evolution equation with a time dependent operator. In view of this new perspective, the direct consequence of the abstract theory is rich enough  to establish global existence of the system.
\abs
\end{abstract}
%
%
%
%
%
%
\section{Introduction}\label{section_introduction}
Consider the following initial-boundary value problem for the fully parabolic system:
\begin{equation}
\begin{cases}\label{P}
u_t=\Delta (\gamma (v)u)&(x,t) \in \Omega\times (0,\infty),\\
v_t=\Delta v-v+u&(x,t) \in \Omega\times (0,\infty),\\
\partial_\nu u=\partial_\nu v=0 \qquad &(x,t) \in \partial\Omega\times (0,\infty),\\
u(x,0)=u_0(x),\;\;v(x,0)=v_0(x),\qquad & x\in\Omega,
\end{cases}
\end{equation}where $\Omega\subset\mathbb{R}^n$ ($n\geq3$) is a smooth bounded domain.
The initial data $(u_0,v_0)$ satisfies
\begin{equation}\label{ini}
 (u_0,v_0)\in (W^{1,p_0}(\Omega))^2 \;\text{with some}\;p_0>n,\quad u_0\geq0,\; v_0 \geq0 \quad  \mbox{in } \overline\Omega, \quad u_0\not\equiv0,
\end{equation}
and for $\gamma$ we assume
\begin{equation}\label{gamma1}
\gamma\in C^3[0,\infty),\;\gamma>0,\;\;\gamma^\prime \leq0\;\;\text{on}\;(0,\infty),
\end{equation}
 and the vanishing property
\begin{equation}\label{gamma2}
\lim\limits_{s\rightarrow\infty}\gamma(s)=0.
\end{equation}


The system \eqref{P} is motivated from biological background (see \cite{KSb, Science, PRL12}). 
From a mathematical view point, this system resembles the fully parabolic Keller--Segel system. 
Especially, in the case $\gamma (v) = e^{-v}$, 
the system \eqref{P} has same mathematical structures as one of the Keller--Segel system (\cite{JW, FJ20, BLT}). 
Indeed, smooth nonegative solutions of \eqref{P} satisfy the mass conservation law
$$\|u(t)\|_{L^1(\Omega)} = \|u_0\|_{L^1(\Omega)}\qquad t>0,$$
and the Lyapunov functional is constructed:
$$
\frac{d}{dt}\mathcal{F}(u(t),v(t)) 
+ \io ue^{-v} |\nabla (\log u - v)|^2\,dx=0 \qquad t>0,
$$
where 
$$
\mathcal{F}(u,v) := 
\io u \log u\,dx - \io uv\,dx + \frac{1}{2}\io v^2\,dx + \io |\nabla v|^2\,dx.
$$
We mention the Keller--Segel system has the same Lyapunov functional (\cite{NSY}). 
Moreover, these systems shares the same stationary problem (\cite{JW, FJ20}) and the scaling structure.  
For the simplified case 
$$
\begin{cases}
u_t=\Delta (e^{-v}u)&(x,t) \in \R^n \times (0,\infty),\\
v_t=\Delta v +u&(x,t) \in \R^n \times (0,\infty),
\end{cases}
$$
the scaled function 
$(u_\lambda(x,t), v_\lambda(x,t)) : = (\lambda^2 u(\lambda x,\lambda^2 t), v(\lambda x, \lambda^2 t))$ ($\lambda>0$) satisfies the simplified system and 
$$
\| u_\lambda \|_{L^1(\R^n)}
=
\lambda^{2-n}
\| u \|_{L^1(\R^n)}.
$$
Thus the mass critical dimension is $n=2$. 
Actually, in \cite{JW, FJ20} the critical mass phenomenon is observed in the two dimensional setting:
\begin{itemize}
\item for small mass, solutions exists globally and remains bounded uniformly in time;
\item  for large mass, solutions exists globally and may blow up at infinite time.
\end{itemize}

In the present paper, we focus on the supercritical case ($n\geq3$). In \cite{BLT}, for the case $\gamma(v)=e^{-v}$ global existence of very weak solutions is established in arbitrary dimensional setting. 
Global existence and boundedness of solutions to \eqref{P} with some polynomial decaying function $\gamma$ are obtained under smallness conditions on some parameters in \cite{YK2017, D2019, FJ21}. 
For the simplified case (the second equation of \eqref{P} is replaced by the elliptic equation), in \cite{JL} global existence is established for any magnitude of mass in the higher dimensional setting. 
On the other hand, the following theorem claims global existence of classical solutions for the fully parabolic system independently of the magnitude of mass.

\begin{theorem}\label{Main}
Let $n\geq 3$.
Assume $\gamma$ satisfies $\eqref{gamma1}$ and $\eqref{gamma2}$.
For any given initial data $(u_0,v_0)$ satisfying \eqref{ini}, The system \eqref{P} permits a unique global classical solution $$(u,v)\in (C^0([0,\infty); W^{1,p_0}(\Omega) )\cap C^{2,1}(\overline{\Omega}\times(0,\infty)))^2.$$
\end{theorem}

Moreover, for $\gamma (v)=e^{-v}$, the higher dimensional case ($n \geq3$) is supercritical as mentioned above. The next theorem guarantees an infinite time blow-up for any magnitude of mass.

\begin{theorem}\label{Main2}
Let $n\geq 3$, $\Omega = B_R(0)$ and $\gamma(s)=e^{-s}$. 
For any $m >0$ there exists some radially symmetric initial datum 
$(u_0,v_0)$ satisfying \eqref{ini} such that
\begin{itemize}
\item $\io u_0= m;$
\item the global classical solution $(u,v)$ of \eqref{P} blows up at infinite time, that is,
	\begin{equation*}
	\limsup\limits_{t\nearrow \infty} 
	\left( \|u(\cdot,t)\|_{L^\infty(\Omega)}
	+\|v(\cdot,t)\|_{L^\infty(\Omega)} \right)=\infty.
	\end{equation*}
\end{itemize}
\end{theorem}
\begin{remark}
For the Keller--Segel system, a finite time blowup phenomenon occurs in the supercritical case (\cite{Win2013}). 
Although the system \eqref{P} shares several features of the Keller--Segel system, solutions of \eqref{P} exist globally even in the supercritical case (Theorem \ref{Main}) and a blow-up phenomenon takes place at infinite time (Theorem \ref{Main2}).
\end{remark}
\section{Preliminaries}\label{section_preliminaries}
We recall some useful lemmas. 
Local existence and uniqueness of classical solutions to system \eqref{P} are established by applying the abstract theory (cf. \cite[Theorem 1]{Horstmann}).
\begin{lemma}\label{local}
Assume that $\gamma$ satisfies \eqref{gamma1} and $(u_0,v_0)$ satisfies \eqref{ini}. 
There exists $T_{\mathrm{max}} \in (0, \infty]$ such that problem \eqref{P} permits a unique classical solution 
$$(u,v)\in  (C^0([0,T_{\mathrm{max}}); W^{1,p_0}(\Omega) )\cap C^{2,1}(\overline{\Omega}\times(0,T_{\mathrm{max}})))^2.$$
If $T_{\mathrm{max}}<\infty$, then
	\begin{equation*}
	\limsup\limits_{t\nearrow T_{\mathrm{max}}} 
	\left( \|u(\cdot,t)\|_{L^\infty(\Omega)}+\|v(\cdot,t)\|_{L^\infty(\Omega)} \right)=\infty.
	\end{equation*}
 Moreover the solution is positive on $ \overline{\Omega}\times(0,T_{\mathrm{max}})$
and the mass conservation law holds:
	\begin{equation*}
	\int_{\Omega}u(\cdot,t)\,dx=\int_{\Omega}u_0 \,dx
	\quad \text{for\ all}\ t \in (0,T_{\mathrm{max}}).
	\end{equation*}	
	\end{lemma}
Let $(u,v)$ be a solution of \eqref{P} in $\Omega \times (0,\tm)$. 
We introduce the auxiliary function $w(x,t)$ such that 
\begin{equation*}
\begin{cases}
-\Delta w+w=u \qquad&(x,t) \in \Omega\times (0,\tm),\\
\partial_\nu w=0&(x,t) \in \partial\Omega\times (0,\tm).
\end{cases}
\end{equation*} 
Here we remark $w = (1-\Delta)^{-1}u \in C^{2,1}(\overline{\Omega}\times(0,T_{\mathrm{max}}))$ by the elliptic regularity theory.
The following lemmas are established in \cite[Lemma 5, Lemma 7]{FJ20}.
\begin{lemma}\label{keylem1}
For any $0<t<T_{\mathrm{max}}$, there holds
	\begin{equation*}
	w_t+\gamma(v)u=(I-\Delta)^{-1}[\gamma(v)u].
	\end{equation*}
	Moreover, for  any $x\in\Omega$ and  $t\in[0,T_{\mathrm{max}})$, it follows
	\begin{equation*}
		0 \leq w(x,t)\leq w_0(x)e^{\gamma(0)t}.
	\end{equation*}
\end{lemma}
\begin{lemma}\label{keylem2} 
There exists $K>0$ depending on the initial data and $\gamma$ such that for all $(x,t)\in\Omega\times[0,T_{\mathrm{max}})$,
	\begin{equation*}
	v(x,t)\leq K\left(w(x,t)+1\right).
	\end{equation*} 
\end{lemma}
By Lemma \ref{keylem1} and Lemma \ref{keylem2}, 
it follows for any $T>0$  
\begin{equation}\label{v_bound}
0\leq v(x,t) \leq v^\ast(T):=K(w_0(x)e^{\gamma(0)T}+1) \qquad \mbox{for all }x\in\Omega, t\in(0,T).
\end{equation}
\section{Proof of Theorem \ref{Main}}\label{section_proof}
\begin{lemma}\label{lem_hoelder} 
Let $(u,v)$ be a solution of \eqref{P} in $\Omega \times (0,\tm)$. 
If $\tm<\infty$, then there exists some $\alpha\in (0,1)$ such that 
$$v \in C^{\alpha, \frac{\alpha}{2}} ({\overline{\Omega}\times [\tfrac{1}{2}\tm,\tm]}).$$
\end{lemma}
\begin{proof}
As noted in Lemma \ref{keylem1}, $(u,v,w) \in (C^{2,1}(\overline{\Omega}\times(0,T_{\mathrm{max}})))^3$ satisfies
\begin{equation}\label{key_eq1}
w_t- \gamma(v)(\Delta w -w)=(I-\Delta)^{-1}[\gamma(v)u]
\end{equation}
classically on $(0,\tm)$. 
Here functions $\gamma (v)$ and $(I-\Delta)^{-1}[\gamma(v)u]$ are bounded on $[\frac{1}{2}\tm,\tm]$.  
Indeed, it follows by \eqref{v_bound} that for all $t \in [\frac{1}{2}\tm,\tm)$,
$$\gamma (v^\ast (\tm)) \leq \gamma(v(t)) \leq \gamma (0).$$
By the maximum principle we have
\begin{eqnarray*}
0\leq  (I-\Delta)^{-1}[\gamma(v)u] 
 &\leq& (I-\Delta)^{-1}[\gamma(0)u]\\ 
&=& \gamma(0) (I-\Delta)^{-1}[u]\\
&=& \gamma(0) w,
\end{eqnarray*}
and then by Lemma \ref{keylem1} there exists some $C(\tm)>0$ such that for all $t \in (0,\tm)$,
\begin{equation}\label{inftyest_z}
0\leq  (I-\Delta)^{-1}[\gamma(v)u] \leq C(\tm).
\end{equation}
Since $w$ is a bounded weak solution of \eqref{key_eq1} on $[\frac{1}{2}\tm,\tm]$ with $w(0)= w(\frac{1}{2}\tm)$,  
the H\"older regularity estimate (\cite[Corollary 7.51]{Lieberman}) guarantees some $\alpha\in (0,1)$ such that 
$$w \in C^{\alpha, \frac{\alpha}{2}} ({\overline{\Omega}\times [\tfrac{1}{2}\tm,\tm]}).$$

Let $z(x,t):= (1-\Delta)^{-1}v(x,t). $
By the definition of $w$ and $z$, it follows
\begin{equation}\label{key_eq2}
z_t = \Delta z -z +w,
\end{equation}
where $z$ is bounded and $w \in C^{\alpha, \frac{\alpha}{2}} ({\overline{\Omega}\times [\tfrac{1}{2}\tm,\tm]})$ as proved above. 
By applying the parabolic regularity estimate (\cite[Chapter IV, Theorem 5.3]{LSU})  to \eqref{key_eq2}, we have 
$$z \in C^{2+\alpha, 1+\frac{\alpha}{2}} ({\overline{\Omega}\times [\tfrac{1}{2}\tm,\tm]}),$$
where we used the fact $z(\tfrac{1}{2}\tm)$ satisfies the compatibility condition. 
Thus by the elliptic regularity estimate it follows
$$v \in C^{\alpha, \frac{\alpha}{2}} ({\overline{\Omega}\times [\tfrac{1}{2}\tm,\tm]}).$$
We conclude the proof.
\end{proof}

\begin{proof}[Proof of Theorem \ref{Main}]
Assume to the contrary that $\tm<\infty$.
Define the operator
$$
A(t) \varphi := -\gamma (v(t+\tfrac{1}{2}\tm)) (\Delta \varphi -\varphi)
$$
for $\varphi \in \{ \psi \in  W^{2, p}(\Omega) \,|\, \partial_\nu \psi = 0\ \mbox{on }\partial \Omega\}$ with $p\in (1,\infty)$. 
Since $\gamma (v) $ is H\"older continuous (Lemma \ref{lem_hoelder}), 
the operator $-A(t)$ generates analytic semigroup in $L^p(\Omega)$ for any $p\in (1,\infty)$. 
Denoting
\begin{eqnarray*}
W(s) &:=& w(s+\tfrac{1}{2}\tm),\\
F(s) &:=& 
(1-\Delta)^{-1}(\gamma(v(s+\tfrac{1}{2}\tm))u(s+\tfrac{1}{2}\tm))
\end{eqnarray*}
for $s \in [0,\tfrac{1}{2}\tm)$, 
we can check that $W \in C([0,\tfrac{1}{2}\tm); L^p(\Omega)) \cap C^1((0,\tfrac{1}{2}\tm); L^p(\Omega))$ is a solution of the evolution equation 
$$
\begin{cases}
\frac{d}{dt} W + A(t)W = F \qquad t \in [0,\tfrac{1}{2}\tm),\\
W(0)=w(\tfrac{1}{2}\tm).
\end{cases}
$$
Since the elliptic regularity theorem implies $F \in C([0,\tfrac{1}{2}\tm); L^p(\Omega))$ for any $p \in (1,\infty)$, 
in view of the abstract theory (\cite[Theorem 5.2.2]{Tanabe}), 
the solution $W$ can be represented by the integral equation: 
$$
W(t) = U(t,0) w(\tfrac{1}{2}\tm) + \int_0^t U(t,s)F(s)\,ds
\qquad t \in [0,\tfrac{1}{2}\tm),
$$
where $U(t,s)$ is the fundamental solution. 
Due to the fact that $w(\tfrac{1}{2}\tm)$ and $F(s)$ satisfy the Neumann boundary condition, we can apply the estimate of fundamental solutions (\cite[Theorem 5.2.1]{Tanabe}) to have 
\begin{eqnarray*}
&&\| A(t)W(t)\|_{L^p(\Omega)} \\
&\leq & \| A(t) U(t,0) (A(0))^{-1}A(0)w(\tfrac{1}{2}\tm) \|_{L^p(\Omega)}
 + \int_0^t \| A(t) U(t,s)(A(s))^{-1}A(s)F(s)\|_{L^p(\Omega)} \,ds\\
 &\leq & C(p) \| A(0)w(\tfrac{1}{2}\tm) \|_{L^p(\Omega)}
 + C(p) \int_0^t  \| A(s)F(s)\|_{L^p(\Omega)} \,ds
\end{eqnarray*}
with some $C(p)>0$. 
By the definition of $W$, it follows
\begin{eqnarray*}
\| A(t)W(t)\|_{L^p(\Omega)} 
&= & \| \gamma (v(t+\tfrac{1}{2}\tm)) (\Delta w(t+\tfrac{1}{2}\tm) -w(t+\tfrac{1}{2}\tm))\|_{L^p(\Omega)} \\
&=&
\| \gamma (v(t+\tfrac{1}{2}\tm))u(t+\tfrac{1}{2}\tm)\|_{L^p(\Omega)}.
\end{eqnarray*}
We also obtain for $s \in [0,\tfrac{1}{2}\tm)$,
\begin{eqnarray*}
A(s)F(s) &=& 
-\gamma (v(s+\tfrac{1}{2}\tm)) (\Delta - 1)
\left[(1-\Delta)^{-1}(\gamma(v(s+\tfrac{1}{2}\tm))u(s+\tfrac{1}{2}\tm))\right]\\ 
&=&  (\gamma (v(s+\tfrac{1}{2}\tm)))^2 u(s+\tfrac{1}{2}\tm).
\end{eqnarray*}
Since $\gamma (v)$ is bounded, for any $p\in (1,\infty)$ it follows for $t\in [0,\tfrac{1}{2}\tm)$
$$
 \| u(t+\tfrac{1}{2}\tm)) \|_{L^p(\Omega)}  
 \leq 
 C^\prime(p) \|(\Delta-1) w(\tfrac{1}{2}\tm)\|_{L^p(\Omega)}
 + C^\prime(p) \int_0^t  \| u(s+\tfrac{1}{2}\tm) \|_{L^p(\Omega)} \,ds
$$
with some $C^\prime(p)>0$.
By the Gronwall inequality, for any $p\in (1,\infty)$ 
there exists some $ C^{\prime\prime}(p)>0$ such that
$$
\limsup_{t \to \tm } \|  u(t)\|_{L^p(\Omega)}
 \leq  C^{\prime\prime}(p).
$$
We can pick up some $p>\frac{n}{2}$ in the above and by Moser's iteration argument (see \cite{Anh19}) it follows
$$
\limsup_{t \to \tm } \| u(t)\|_{L^\infty(\Omega)} < \infty,
$$
which contradicts the assumption $\tm<\infty$.  
\end{proof}  
As mentioned in Introduction, the Lyapunov functional and the stationary problem of \eqref{P} with the case $\gamma (v) = e^{-v}$ are same as one of the Keller--Segel system. 
We can directly employ calculations in \cite[Section 3]{Win2010}.
\begin{proof}[Proof of Theorem \ref{Main2}]
On the contrary, the solution $(u,v)$ is uniformly bounded in time. 
By the parabolic Schauder theory and the compactness argument, there exist a sequence of time $t_k \to \infty$ and a stationary solution $(u_\infty, v_\infty)$ such that the solution $(u(\cdot,t_k),v(\cdot,t_k))$ converges to $(u_\infty, v_\infty)$ in $C^2(\overline{\Omega})$ as $t_k \to \infty$, where the following inequality holds
\begin{equation}\label{contra}
\mathcal{F} (u_\infty, v_\infty) \leq \mathcal{F}(u_0,v_0)
\end{equation}
(see \cite[Lemma 3.1]{Win2010}).
Through all radially symmetric stationary solutions $(u_\infty, v_\infty)$, the infimum of the functional $\mathcal{F}(u_\infty, v_\infty)$ is bounded from below (\cite[Lemma 3.4]{Win2010}).  
On the other hand, for any $m>0$ we could construct an initial datum $(u_0,v_0)$ with$\io u_0 =m$ having sufficiently large negative energy 
$\mathcal{F}(u_0,v_0)$ (\cite[Lemma 3.2]{Win2010}), 
which contradicts \eqref{contra}. 
Therefore the corresponding global solution $(u,v)$ must be unbounded.
\end{proof}
\bigskip
\noindent\textbf{Acknowledgments} \\
K. Fujie is supported by Japan Society for the Promotion of Science (Grant-in-Aid for Early-Career Scientists; No.\ 19K14576). 
T. Senba is supported by Japan Society for the Promotion of Science (Grant-in-Aid for Scientific Research(C); No.\ 18K03386)

\end{document}